\documentclass[11pt]{amsart}
\usepackage{amssymb}
\usepackage{cite}				
\usepackage{graphicx}

\newtheorem{theorem}{Theorem}[section] 
\newtheorem{proposition}{Proposition}[section] 
\newtheorem{lemma}{Lemma}[section]

\theoremstyle{remark}
\newtheorem{remark}{Remark}
\newtheorem{example}{Example}
 
\newcommand{\B}{\mathcal{B}}

\newcommand{\Id}{\mbox{Id}} 
\newcommand{\R}{\mathbb{R}}

\newcommand{\bo}{\partial\Omega} 
 
\renewcommand{\d}{\mathrm{d}} 

\author{Carlos Montalto }
\author{Plamen Stefanov}
\address{Department of Mathematics, Purdue University, West Lafayette, IN 47907}
\thanks{Both authors partly supported by NSF,  Grant DMS-0800428}
\title{Stability of Coupled-Physics Inverse Problems with  internal measurements}
\date{June 8, 2013}

\begin{document}

\begin{abstract}
In this paper, we develop a general approach to prove stability for the non linear second step of hybrid inverse problems.  We work with general functionals of the form $\sigma|\nabla u|^p$, $0 < p \leq 1$, where $u$ is the solution of the elliptic partial differential equation $\nabla\cdot \sigma \nabla u =0$ on a bounded domain $\Omega$ with boundary conditions $u|_{\partial \Omega} = f$. We prove stability of the linearization and H\"older conditional stability for the non-linear problem of recovering $\sigma$ from the internal measurement.
\end{abstract}

\maketitle
\section{Introduction}
Couple-physics Inverse Problems or Hybrid Inverse Problems is a research area that is interested in developing  the mathematical framework for medical imaging modalities that combine the best imaging properties of different types of waves (e.g., optical waves, electrical waves, pressure waves, magnetic waves, shear waves, etc) \cite{Ammari2008, Bal2004introduction, Bal2011hybrid, Wang2012biomedical}. In some applications of non-invasive medical imaging modalities (e.g., cancer detection) there is  need for high contrast and high resolution images. High contrast discriminates between healthy and non-healthy tissue whereas high resolution is important to detect anomalies at and early stage \cite{Bal2012hybrid}. In some situation current methodologies (e.g., electrical impedance tomography, optical tomography, ultrasound, magnetic resonance) focus only in a particular type of wave that can either recover high resolution or high contrast, but not both with the required accuracy. For instance, electrical impedance tomography (EIT) and optical tomography (OT) are high contrast modalities because they can detect small local variations in the electrical and optical properties of a tissue. However because of their high instability they are characterized by their low resolution images \cite{Borcea2002EIT, Cheney-Isaacson-Newel1999EIT}. On the other hand, ultrasound tomography and magnetic resonance imaging are modalities that provide high resolution but not necessarily high enough contrast since the difference between the index of refraction of the healthy and non-healthy tissue is very small \cite{Bal2012hybrid}. 

The aim of hybrid inverse problems is to couple the physics of each wave to benefit from the imaging advantages of each one. Some examples of this physical coupling are: (i) ultrasound modulated electrical impedance tomography (UMEIT) also known as acoustic-electro tomography (AET) or electro acoustic tomography (EAT) \cite{Ammari2008, Ammari-Bonnetier-Capdeboscq-Tanter-Fink2008, Gebauer-Scherzer2008UMEIT, Kuchment-Kuyansky2010inverse, Kuchment-Kuyansky2011UMEIT}; (ii) current density impedance imaging (CDII) \cite{Nachman-Tamasan-Timonov2011MRIEIT,Nachman-Tamasan-Timonov2009MRIEIT,Nachman-Tamasan-Timonov2007conductivity, DBLP:journals/tmi/HasanovMNJ08}; and (iii) ultrasound modulated optical tomography (UMOT) also known as acoustic optical tomography (AOT) \cite{Allmaras-Bangerth2011,Bal2012, BalMoskow2013local, BalSchotland2010inverse, Nam2002ultrasound}.

All of these hybrid inverse problems involve two steps. In the first step the high resolution modality takes an input boundary measurements $f$ and provides an output internal functional of the form $\sigma|\nabla u|^p$ for $p>0$, where $u$ is the solution of the elliptic partial differential equation $\nabla\cdot \sigma \nabla u =0$ on a bounded domain $\Omega$ with boundary conditions $u|_{\partial \Omega} = f$. Physically, $\sigma$ is the unknown conductivity (or diffusion coefficient) and $u$ is the electric potential (or photon-density) of the tissue, depending on whether we are looking for electrical (or optical) properties of the tissue. In the second step the high contrast modality recovers the conductivity (or diffusion coefficient) $\sigma$ from the knowledge of the internal functional $\sigma|\nabla u|^p$ for $p>0$. Different values of $p$ represent different physical couplings, in the case of CDII, $p$ equals 1, and in the case of UMEIT and UMOT, $p$ equals $2$. Other internal functionals have been studied as well \cite{bal2013local}.

In this paper we develop a general approach to prove stability for the non linear second step of  these hybrid inverse problems.  We work with general functionals of the form $\sigma|\nabla u|^p$, $0 < p \leq 1$. We prove stability of the linearization, and H\"older conditional stability for the non-linear problem. In the appendix, we generalize the abstract stability approach in  \cite{Stefanov20092842} to transfer conditional stability of the linearization to conditional stability of the non-linear problem.  The behavior of the linearized problem depends on whether $ 0 < p< 1$, $p =1$, or $p>1$ as has been noted before, see, e.g., \cite{Kuchment-Steinhauer2012, Bal2012hybrid}. The case $0<p<1$ is the simplest one since the linearized operator becomes elliptic and thus stable. When $p=1$, the linearized operator can be considered as one parameter family of elliptic operators on a family of hypersurfaces allowing us to show stability by superposition of elliptic operators. Finally, when $p>1$ the linearized operator becomes hyperbolic, see also \cite{Bal2012hybrid}. For completeness in the exposition we analyze the case $0<p<1$ as well even though it does not appear in applications to medical imaging.

A  unified manner of dealing with the linearization of this problem was proposed in \cite{Kuchment-Steinhauer2012}, for the cases  $0<p<1$ and $1\leq p \leq 2$. In the first case they used  one measurement, while in the second one, they required two  measurements. In both cases they prove that the linearization  is elliptic in the interior of the domain. This  implies stability of the linearized problem, up to a finite dimensional kernel, without necessarily having injectivity. The     conductivity  $\sigma$ in \cite{Kuchment-Steinhauer2012} is perturbed by functions $\delta\sigma$  identically zero in a fixed neighborhood of the boundary. We allow perturbations in the whole domain, with appropriate boundary conditions. We use one boundary measurement even in the case $p=1$ (CDII). For $0< p \leq 1$, we show stability, and hence injectivity, for the non-linear problem and its linearization. Our approach is based on a factorization of the linearization, see \eqref{eq:EllipticBoundProb} below. 
 Instead of analyzing the linearization using the pseudo-differential calculus, we analyze the only non-trivial factor in the factorization, which happens to be a second order differential operator.  

In the specific cases of $p=1$ and $p=2$, this hybrid inverse problems had been largely studied. For the case $p=1$, inversion procedures and reconstruction were obtained in \cite{Nachman-Tamasan-Timonov2011MRIEIT, Nachman-Tamasan-Timonov2009MRIEIT, Nachman-Tamasan-Timonov2007conductivity}. In the case $p=2$ with several measurements, a numerical approach was proposed in \cite{Capdeboscq-Fehrenbach-Gournay2009imaging} in $C^{1,\alpha}$ for conductivities zero near the boundary  and in \cite{Bal-Bonnetier-Monard-Faouzi2011inverse}, a global estimate was established in $W^{1,\infty}$. 

\subsection{Main results}

Let $\Omega$ be a bounded simply connected open set of $\R^n$ with smooth boundary. Consider the strictly elliptic boundary value problem 
\begin{equation}\label{eq:EllipticBoundProb}
\nabla\cdot\sigma \nabla u=0\quad \text{in } \Omega, \qquad u|_{\bo}=f,
\end{equation}
where $\sigma$ is a function in $C^2(\overline{\Omega})$  such that $\sigma >0$ in $\overline{\Omega}$ and $f \in C^{2,\alpha}(\partial \Omega)$, $0<\alpha<1$. By the Schauder estimates,   $u \in C^2(\overline{\Omega})$. We say $u$ is $\sigma-$harmonic if it satisfies equation \eqref{eq:EllipticBoundProb}. We address the question of whether we can determine $\sigma$, in a stable way, from the functional $F:C^2(\overline{\Omega}) \to C(\bar \Omega)$ defined by 
\[
F(\sigma) = \sigma|\nabla u|^p,
\]
with  $p>0$ is fixed. This problem has different behavior depending on whether $0<p<1$, $p=1$ or $p >1$.

We study stability of the non-linear problem by proving first stability for the linearization, see section~\ref{sec_2}, and then using Theorem \ref{non-linear_stability_theorem}. The latter  is a generalization of the main result in \cite{Stefanov20092842}, that allows to obtain stability for the non-linear problem from stability of the linearized problem. Our main theorem about  stability for the linearized problem is the following. 
\begin{theorem}[Stability of the linearization]\label{theor:cond_stability_linearization}
Let $u_0$ be $\sigma_0-$harmonic with $\nabla u_0\neq 0$ in $\overline{\Omega}$ and let $d_{\sigma_0}F$ be  the differential  of $F$ at $\sigma_0$. 
\begin{itemize}
\item Case $0<p<1$: there exist $C>0$ such that 
\begin{equation}\nonumber 
\|h\| \leq C\|d_{\sigma_0}F(h)\|_{H^1(\Omega)} \quad \mbox{ for   every } 
h\in H_0^1(\Omega); 
\end{equation}
\item Case $p=1$: for any $\alpha_1 \in [0,1)$, there exist $C>0$ such that if $(1-\alpha_1)s_1 \geq 2$
\begin{equation}\label{cond_stability_linearization_p=1}
\|h\| \leq C\|d_{\sigma_0}F(h)\|_{H^1(\Omega)}^{\alpha_1}\|h\|_{H^{s_1}(\Omega)}^{1-\alpha_1} \quad
\mbox{ for  every } h \in H^{s_1}(\Omega)\cap H^1_0(\Omega);
\end{equation}
\end{itemize}
where $\nu(x)$ denotes the outer-normal vector to the boundary.
\end{theorem}

This together with Theorem \ref{non-linear_stability_theorem} gives our main result about stability for the non-linear problem. 
\begin{theorem}[Stability for the non-linear map $F$, case $0<p\leq 1$]\label{theor:cond_stability_main_theorem}
Let $0<p\le1$. 
Let $u_0$ be $\sigma_0-$harmonic with $\nabla u_0\neq 0$ in $\overline{\Omega}$. For any $0< \theta < 1$, there exist $s >0$ so that if $\| \sigma \|_{H^{s}(\Omega)} <L$ for some $L>0$, there exist $\epsilon >0$ such that
\[
\| \sigma - \sigma_0 \|_{C^2(\bar \Omega)} < \epsilon
\]
implies
\begin{equation}\label{ineq:main_theorem_non-linear_stability}
\| \sigma - \sigma_0 \|_{L^2(\Omega)} < C\| F(\sigma) - F(\sigma_0) \|^\theta_{L^2(\Omega)}.
\end{equation}
\end{theorem}

\begin{remark}
In the case of $\R^2$ we can satisfy $\nabla u_0\not=0$ in $\overline{\Omega}$ by imposing conditions on $f$. For instance in \cite{Alessandrini1985} and \cite{Nachman-Tamasan-Timonov-2007} the authors showed if  $\Omega$ is simply connected in $\R^2$, $\sigma_0 \in C^\alpha(\Omega)$ $0< \alpha <1$ and $u_0|_{\bo}$ is continuous and two-to-one map, except possibly at its maximum and minimum. Then $|\nabla u|>0$ in $\overline{\Omega}$. 
\end{remark}

\textbf{Acknowledments.} The authors would like to thank Adrian Nachman for his advice. This work started when the second author was visiting the Fields Institute in Toronto.  

\section{Linearization}\label{sec_2} 
We start by considering the linearized version of this problem. Denote by $\d F_{\sigma_0}$ the G\^ateux derivative of $F$ at some fixed $\sigma_0$. For $\sigma$ in a $C^2$-neighborhood of $\sigma_0$ we get 
\begin{equation} \label{2-order_linerization}
F(\sigma) = F(\sigma_0) + \d F_{\sigma_0}(\sigma - \sigma_0) + \int_0^1 (1-t)\d^2F_{\sigma_0 + t(\sigma - \sigma_0)}(\sigma - \sigma_0,\sigma - \sigma_0) \d t
\end{equation} 
where $\d F_{\sigma_0}$ is given by 
\begin{equation}\label{dF}
\d F_{\sigma_0}(h) =   h|\nabla u_0|^p  + p |\nabla u_0|^{p-2} \sigma_0 \nabla u_0\cdot \nabla v_0(h)
\end{equation} 
and $\d^2 F_{\sigma_t }$ by 
\begin{equation}\label{d2F}
\begin{split}
\d^2 F_{\sigma_t}(h,h) 
& = p|\nabla u_t|^{p-2} \left( h \nabla u_t   \cdot \nabla v_t(h) + \nabla v_t(h) \cdot \nabla v_t(h) +  \nabla u_t   \cdot \nabla w_t(h)\right)  \\
& \qquad \qquad \qquad \qquad \qquad + p(p-2)|\nabla u_t |^{p-4} (\nabla u_t   \cdot \nabla v_t(h))^2,
\end{split}
\end{equation} 
for $h = \sigma - \sigma_0 \in C^2(\overline{\Omega})$ and $\sigma_t  = \sigma_0 + t(\sigma -\sigma_0) $ for $0 \leq t \leq 1$ and $u_t$, $v_t$ and $w_t$ solving
\begin{equation}\label{u_v_w}
\begin{aligned}
\nabla\cdot {\sigma_t}\nabla u_t & = 0 \\ 
\nabla\cdot {\sigma_t}\nabla v_t   &=  - \nabla\cdot h\nabla u_t \\
\nabla\cdot {\sigma_t}\nabla w_t   &= - 2 \nabla\cdot h\nabla v_t
\end{aligned}
\quad
\begin{aligned}
\mbox{in }\Omega,\\
\mbox{in }\Omega,\\
\mbox{in }\Omega,
\end{aligned}
\quad
\begin{aligned}
u_t|_{\bo}&=f; \\
v_t|_{\bo}&=0; \\
w_t|_{\bo}&=0; 
\end{aligned}
\end{equation}
for $0 \leq t \leq 1$. 

Let 
\[
R_{\sigma_0}(h) = \int_0^1 (1-t)\d^2F_{\sigma_0 + th}(h,h) \d t \quad \forall h \in C^2(\overline{\Omega}),
\]
we claim that 
\begin{equation} \label{second_order_residue}
\|R_{\sigma_0}(h)\| \leq C_{\sigma_0} \| h\|^2_{C^2(\Omega)}
\end{equation}
where 
\[
C_{\sigma_0} = C\sup_{0\leq t\leq 1} \left((2p+1) \|\nabla u_t \|_{C^2(\overline{\Omega})}^p + p(p-2)\| \nabla u_t \|^{2p-2}_{C^2(\overline{\Omega})}\right)
\]
 with $C$ depending only on $\Omega$ and the dimension $n$. Assuming the claim then  $dF_{\sigma_0}$ is a linearization of $F$ at $\sigma_0$ with a quadratic remainder as in (\ref{alpha_order_linearization}). 

To show (\ref{second_order_residue}) we estimate (\ref{d2F}) using inequalities \eqref{ineq:vt_estimate} and \eqref{ineq:wt_estimate}. These last two inequalities are consequence of \eqref{u_v_w} and elliptic regularity \cite{GT-2001}. Let $C>0$ be a constant depending on $\Omega$ and the dimension $n$, using the convention that $C$ can increase from step to step we have 
\begin{equation} \label{ineq:vt_estimate}
\begin{split}
\|\nabla v_t\|_{C^{1,\alpha}(\overline{\Omega})} & \leq \|v_t\|_{C^{2,\alpha}(\overline{\Omega})} \quad \mbox{ for } \quad \alpha \in (0,1) \\
&\leq C \|\nabla \cdot h \nabla u_t\|_{C^{0,\alpha}(\overline{\Omega})} \leq C \| h \nabla u_t\|_{C^{1,\alpha}(\overline{\Omega})}  \quad \mbox{ for } \quad \alpha \in (0,1) \\
& \leq C \|h\|_{C^2(\overline{\Omega})} \cdot \|\nabla u_t\|_{C^2(\overline{\Omega})},
\end{split}
\end{equation}
and
\begin{equation}\label{ineq:wt_estimate}
\begin{split}
\|\nabla w_t\|   &\leq C\|\nabla w_t\|_{C^{1,\alpha}(\overline{\Omega})} \leq \|w_t\|_{C^{2,\alpha}(\overline{\Omega})} \quad \mbox{ for } \quad \alpha \in (0,1) \\
&\leq C \|\nabla \cdot h \nabla v_t\|_{C^{0,\alpha}(\overline{\Omega})} \leq C \| h \nabla v_t\|_{C^{1,\alpha}(\overline{\Omega})}  \quad \mbox{ for } \quad \alpha \in (0,1) \\
& \leq C \|h\|^2_{C^2(\overline{\Omega})} \cdot \|\nabla u_t\|_{C^2(\overline{\Omega})},
\end{split}
\end{equation}
where the last inequality follows by \eqref{ineq:vt_estimate}.

\subsection*{Decomposition of the Linearization}
We decompose the linearization (\ref{2-order_linerization}) and describe the geometry of $dF_{\sigma_0}$ in more detail in the following two propositions. This analysis holds for any $p>0$.

\begin{proposition}\label{Prop:LinearDecomp}
Let $u_0$ be $\sigma_0$-harmonic with $\nabla u_0 \neq 0$ in $\overline{\Omega}$,  then
\begin{equation} \label{LinearDecomp}
\sigma_0 T_0\frac{\d F_{\sigma_0}(\rho)}{\sigma_0|\nabla u_0|^p } = -L \Delta_{\sigma_0,D}^{-1}T_0\rho \quad \mbox{ for } \quad \rho = (\sigma - \sigma_0)/\sigma_0 \in C^2(\overline{\Omega}),
\end{equation}
where $T_0 = \nabla u_0 \cdot\nabla$ is a transport operator along the gradient field of $u_0$, $\Delta_{\sigma,D}$ is the Dirichlet realization of $\Delta_\sigma:= \nabla\cdot\sigma\nabla$ in $\Omega$ and $L$ is a differential operator given by
\[
Lv := -\nabla\cdot {\sigma_0} \nabla v + p\nabla\cdot\left( {\sigma_0} \frac{\nabla u_0\cdot \nabla v }{|\nabla u_0|^2}   \nabla u_0 \right).
\]
\end{proposition}
\begin{proof}
Since $\nabla u_0 \neq 0$ in $\overline{\Omega}$ we can write (\ref{dF}) as
\begin{equation}\label{dF_rho}
\d F_{\sigma_0}(\rho) =   {\sigma_0}|\nabla u_0|^p \left( \rho + p\frac{\nabla u_0\cdot \nabla v_0(\rho)}{|\nabla u_0|^2} 	\right).
\end{equation} 
Solving (\ref{dF_rho}) for the free $\rho$ term and plugging that into the second equation in  (\ref{u_v_w}) we get
\[
Lv_0  = 
 \nabla\cdot\left(  \frac{\d F_{\sigma_0}(\rho)}{|\nabla u_0|^p } \nabla u_0\right) 
\quad \text{in $\Omega$}, \qquad v_0|_{\bo}=0.
\]
The solution $v_0$ of the second equation in (\ref{u_v_w}) satisfies 
\[
\nabla \cdot \sigma_0 \nabla v_0 = -\nabla \cdot (\sigma - \sigma_0) \nabla u_0 = -\nabla \rho \cdot\nabla u_0 
\] 
and is a linear operator in $\rho$ that can be written as $v_0 = -\Delta^{-1}_{\sigma_0,D}T_0\rho $. So we get
\[ 
-L\Delta^{-1}_{\sigma_0,D}T_0\rho  = \nabla\cdot\left(  \frac{\d F_{\sigma_0}(\rho)}{|\nabla u_0|^p } \nabla u_0\right)=  \sigma_0\nabla u_0 \cdot\nabla\left(  \frac{\d F_{\sigma_0}(\rho)}{\sigma_0|\nabla u_0|^p }\right).
\]
\end{proof}

Notice that in the l.h.s.\ of \eqref{LinearDecomp}, the only non-trivial operator in terms of injectivity is the second order differential operator $L$. We focus our attention on understanding this operator. Denote by $\Pi_0 \omega = (\nabla u_0\cdot \omega/| \nabla u_0 |^2)\nabla u_0   $ the  orthogonal projection of the covector $\omega$  onto $\nabla u_0$ in the Euclidean metric. 
Then $\Pi_\perp:=\Id-\Pi_0$ is the orthogonal projection on the orthogonal complement of $\nabla u_0$.
Take a test function $\phi\in C_0^\infty(\Omega)$, and compute
\begin{equation} \label{eq:L_form}
\begin{split}
(Lv,\phi)& =(\sigma_0\nabla v,\nabla \phi)-  p (\sigma_0\Pi_0\nabla v,\nabla \phi ),\\ 
&= \left( \sigma_0\Pi_\perp\nabla  v,\Pi_\perp\nabla  \phi \right)+(1-p)\left( \sigma_0\Pi_0\nabla  v,\Pi_0\nabla  \phi \right)  .
\end{split}
\end{equation}
We therefore get
\begin{equation*}
L = (\Pi_\perp\nabla)'\cdot\sigma_0(\Pi_\perp\nabla) + (1-p)(\Pi_0\nabla)'\cdot\sigma_0(\Pi_0\nabla),
\end{equation*}
where the prime stands for transpose in distribution sense. 

\begin{example} $\sigma_0=1$, $f=x^n$. Then $u_0=x^n$ and $-L=\Delta_{x'}+(1-p)\partial_{x^n}^2$, where $x=(x',x^n)$. Notice that for $ 0 \leq p <1, L$ is an elliptic operator; for $p=1, L$  becomes the restriction of the Laplacian on the planes $x^n=\text{const.}$; and for $p>1, L$ is a hyperbolic operator.  
\end{example}

Motivated by this example we find a local representation for $L$. We use the convention that Greek superscripts and subscripts run from $1$ to $n-1$. 
\begin{proposition}
Let $u_0\in C^2(\overline{\Omega})$ be $\sigma_0$-harmonic, with $\nabla u_0(x_0) \neq 0$ for $x_0 \in \Omega$. There exist local coordinates $(y',y^n)$ near $x_0$ such that 
\begin{equation} \label{LocalCoordLevelCurve}
dx^2 =    c^2(\d y^n)^2 +  g_{\alpha\beta}\d y^\alpha \d y^\beta, 
\quad g_{\alpha\beta} : = \sum_i \frac{\partial x^i}{\partial y^\alpha}  \frac{\partial x^i}{\partial y^\beta} ,
\end{equation}
where $c= |\nabla u_0|^{-1}$. In this coordinates
\begin{equation}\label{local_representation_L_not_explicit}
 L =  Q   -(1-p)\frac{1}{\sqrt{\det g}} \frac{\partial}{\partial y^n}c^{-2}\sigma_0\sqrt{\det g} \frac{\partial}{\partial y^n},
\end{equation}
where $Q$ is a second order elliptic positively defined differential operator in the variables $y'$ smoothly dependent on $y^n$; in fact, $Q$ is the restriction of $\Delta_{\sigma_0}$ on the level surfaces $u_0=\text{const}$.
\end{proposition}
\begin{proof}
Notice first that $u_0$ trivially solves the eikonal equation $c^2|\nabla \phi|^2=1$ for the speed $c= |\nabla u_0|^{-1}$. Near some point $x_0$, we can assume that $u(x_0)=a$; then $u_0(x)$ is the signed distance from $x$ to the level surface $u_0=a$. Choose local coordinates $y'$ on this level curve, and set $y^n=u_0(x)$. Then $y=(y',y^n)$ are boundary local coordinates to $u_0=a$ and in those coordinates, the metric $c^{-2}\d x^2$ takes the form 
\[
g_{ij}\d x^i \d x^j =    (\d y^n)^2 +c^{-2}  g_{\alpha\beta}\d y^\alpha \d y^\beta, 
\quad g_{\alpha\beta} : = \sum_{i=1}^n \frac{\partial x^i}{\partial y^\alpha}  \frac{\partial x^i}{\partial y^\beta} .
\]
Then
\[
\d x^2 = c^2 (\d y^n)^2 + g_{\alpha\beta}\d y^\alpha \d y^\beta.
\]
Let $\phi \in  C^\infty_0(\Omega)$, using (\ref{eq:L_form}), we get that near $x_0$
\[
\Pi_0\nabla_x = c^{-1}\left(0,\dots,\partial/\partial y^n\right).
\]
Locally near $x_0$ we get,
\begin{equation}
\label{local_representation_L_form}
\begin{split}
(L  v,  \phi ) &= \int \sigma_0 \left( \sum_{i=1}^n \frac{\partial v}{\partial  x^i}    \frac{\partial \bar\phi}{\partial x^i} - p  \frac{\partial v}{\partial y^n} \frac{\partial \bar\phi}{\partial y^n} \right)\d x\\
&= \int\sigma_0 \left(  g^{\alpha\beta}   \frac{\partial v}{\partial y^\alpha}    \frac{\partial \bar\phi}{\partial y^\beta}   +(1-p)  c^{-2} \frac{\partial v}{\partial y^n} \frac{\partial \bar\phi}{\partial y^n}  \right)|\det(\d x/\d y  )|\,\d y.
\end{split}
\end{equation}
Hence
\begin{equation} \nonumber 
L = -\frac{1}{\sqrt{\det g}} \left( \frac{\partial}{\partial y^\beta} \sigma_0  g^{\alpha\beta}  \sqrt{\det g} \frac{\partial}{\partial y^\alpha} + (1-p)\frac{\partial}{\partial y^n}c^{-2}\sigma_0\sqrt{\det g} \frac{\partial}{\partial y^n} \right),
\end{equation}
which proves (\ref{local_representation_L_not_explicit}).
\end{proof}

\begin{remark}
In the two dimensional case we can get an explicit local coordinate system by taking $y^2 = u_0(x)$ and $y^1 = \tilde{u}_0$, with $\tilde{u}_0 \in H^1(\Omega)$ be any the $\sigma_0$-harmonic conjugate of $u_0$, that is $\nabla \tilde{u}_0 = ( \sigma \nabla u_0)^\perp$, where $(a,b)^\perp = (b,-a)$. The level curves of $v_0$ (stream lines) are perpendicular to the level curves of $u_0$ (equipotential lines), see \cite{Astala-Iwaneic-Martin-2009} for details. 
\end{remark}
\begin{remark}
Notice that if $p<1$, $L$ is elliptic (and positive); if $p>1$, $L$ is hyperbolic; and when $p=1$, the operator $L=Q(y^n)$  can be considered as an one parameter family of elliptic operators on the level surfaces of $u_0$.
\end{remark}

\section{Stability estimates}

We first provide a conditional stability estimate for the linearized problem of recovering $\sigma$ from $\sigma|\nabla u|^p$ in \eqref{eq:EllipticBoundProb} for $p>0$. We address this question by using decomposition (\ref{LinearDecomp}). 

The proof of Theorem \ref{theor:cond_stability_linearization} is divided in some lemmas about the stability of the different operator in the decomposition (\ref{LinearDecomp}), we start with the differential operator $L$ 
\begin{lemma} \label{lemma:stability_for_L}
Let $u_0$ be $\sigma_0-$harmonic, with $\nabla u_0\neq 0$ in $\overline{\Omega}$, then
\begin{itemize}
\item Case $0<p<1$: There exist $C>0$ depending on $\sigma$, $n$, $\Omega$ and $u_0$ such that
\begin{equation} \label{ineq:stab-for-L-0<p<1}
\|v\|_{H^2(\Omega)} \leq C \|Lv\|, \quad \mbox{ for } v\in H_0^1(\Omega)\cap H^2(\Omega). 
\end{equation}
\item Case $p=1$: there exist $C>0$ such that
\begin{equation} \nonumber
\|v\|^2_{L^2(\Omega)} \leq C (Lv,v), \quad \mbox{ for } \quad v\in C^\infty(\bar\Omega) \mbox{ with }  v|_{\partial \Omega} =0.
\end{equation}
\end{itemize}
\end{lemma}
\begin{proof}
The proof for the elliptic case $0<p<1$ is an immediate consequence of elliptic theory (see for instance Theorem 8.12 in \cite{GT-2001}) and injectivity of $L$ with Dirichlet boundary conditions. The latter follows from   integration by parts, see \eqref{eq:L_form}. We get that $Lv=0$ with $v=0$ on $\partial \Omega$ implies 
\[
\Pi_\perp \nabla v= \Pi_0\nabla v=0\quad \Longrightarrow\quad \nabla v=0. 
\]
Then $v=0$. 

We now consider the case $p=1$. There exists an open bounded $\Omega_1 $ containing $\overline{\Omega}$  and  a $C^2$ extension of $u_0$ to $\overline{\Omega}_1$ denoted by $u_1$ such that  $\nabla u_1 \neq 0$ on $\overline{\Omega}_1$. We extend $v$ as zero in $\overline{\Omega}_1 \setminus  \Omega$. Let $x_0 \in \overline{\Omega}$, and denote by $\Gamma_0$ the level surface  of $u_1$ in $\overline{\Omega}_1$ containing $x_0$. Clearly $\Gamma_0$ is bounded and closed in $\overline{\Omega}_1$, hence a compact subset of $\R^n$. Its restriction to the interior is an open surface (locally given by $u_0=\text{const.}$ with $\nabla u_0\not=0$). Note that any such level surface may have points on $\partial\Omega$, where it is not transversal to $\partial\Omega$. 

Let $y = (y',y^n)$ be local boundary normal coordinates for $x_0 \in \Gamma_0$ as in (\ref{LocalCoordLevelCurve}). By compactness we can define these coordinates to 
an open neighborhood of $\Gamma_0\cap\overline{\Omega}$ contained in $\Omega_1$. In these  coordinates we can write this open neighborhood as $\tilde{\Gamma}_0\times (a_0 - \epsilon_0, a_0 + \epsilon_0)$, for $\tilde{\Gamma}_0 = \Gamma_0 \cap \tilde{\Omega}$, where $ \Omega \Subset \tilde{\Omega} \Subset \Omega_1 $; $a_0 = u_0(x_0)$; and $\epsilon_0 < \min \{ \mbox{dist}(\partial \Omega, \partial \tilde{\Omega}) , \mbox{dist}(\partial \tilde{\Omega}, \partial \Omega_1) \}$. Using representation (\ref{local_representation_L_form}), ellipticity of \eqref{eq:EllipticBoundProb}, and Poincar\' e  inequality on $\tilde{\Gamma}_0$, we see that for each $x_0\in \overline{\Omega}$ there exist $\epsilon_0$ such that for all $0< \epsilon <\epsilon_0$
\begin{equation}\label{local_estimate_L}
\begin{split}
\int\limits_{a_0-\epsilon}^{a_0+\epsilon} \int\limits_{\tilde{\Gamma}_0} Lv \overline{v}  |\det(\d x/\d y  )|\, \d y'\d y^n 
& = \int\limits_{a_0-\epsilon}^{a_0+\epsilon} \int\limits_{\tilde{\Gamma}_0} \sigma_0   g^{\alpha\beta} \frac{\partial v}{\partial y^\alpha} \frac{\partial \bar v}{\partial y^\beta}|\det(\d x/\d y  )| \,\d y'\d y^n\\ 
&\ge \frac{1}{C}\int\limits_{a_0-\epsilon}^{a_0+\epsilon} \int\limits_{\tilde{\Gamma}_0} |\nabla_{y'} v(y',y^n)|^2\d y'\,\d y^n\\
&\geq \frac{1}{C}  \int\limits_{a_0-\epsilon}^{a_0+\epsilon} \int\limits_{\tilde{\Gamma}_0} |v(y',y^n)|^2\d y'\,\d y^n \geq \frac{1}{C} \|v\|_{L^2(\tilde{\Gamma}_0 \times (a_0 -\epsilon, a_0 +\epsilon))}.
\end{split}
\end{equation}

By compactness of $\overline{\Omega}$ we can find finitely many neighborhoods of level curves of $u_0$, such that (\ref{local_estimate_L}) holds in each of them and their union contains $\overline{\Omega}$, since (\ref{local_estimate_L}) holds for all $0<\epsilon < \epsilon_0$ we can take them to be disjoint. Adding all this estimates we prove the lemma in the $p=1$ case as well. 
\end{proof}

\begin{lemma} \label{lemma:estimate_T0}
Let $u_0$ be $\sigma_0-$harmonic, with $\nabla u_0\neq 0$ in $\overline{\Omega}$, then there exist $C>0$ depending on 
$u_0$ and $\Omega$ such that
\begin{equation}\label{ineq:lemma_T0_stability}
\|h\| \leq C \| \nabla u_0 \cdot \nabla h\| \quad \mbox{ for } \quad h|_{\partial\Omega} =0,
\end{equation}
where $\nu(x)$ denotes the outer-normal vector to the boundary. 
\end{lemma}
\begin{proof}

There exist an open bounded $\Omega_1 $ containing $\overline{\Omega}$  and  a $C^2$ extension of $u_0$ to $\overline{\Omega}_1$ denoted by $u_1$ such that  $\nabla u_1 \neq 0$ on $\overline{\Omega}_1$. We extend $h$ as zero in $\overline{\Omega}_1 \setminus\Omega$. This extension commutes with the differential because $h=0$ on $\partial\Omega$. 
Let $x_0 \in \overline{\Omega}$, denote by $\Gamma_0$  the level surface of $u_1$ in $\overline{\Omega_1}$ containing $x_0$. We work in $y = (y', y^n)$, local boundary normal coordinates for $x_0= (y'_0,y^n_0)$ as in (\ref{LocalCoordLevelCurve}). Notice that since $\nabla u_1 \neq 0$ in $\overline{\Omega_1}$, these coordinates can be extended through the integral curves of the gradient field of $u_0$. 

\begin{figure}[h!] 
  \centering
  \includegraphics[scale=0.3]
{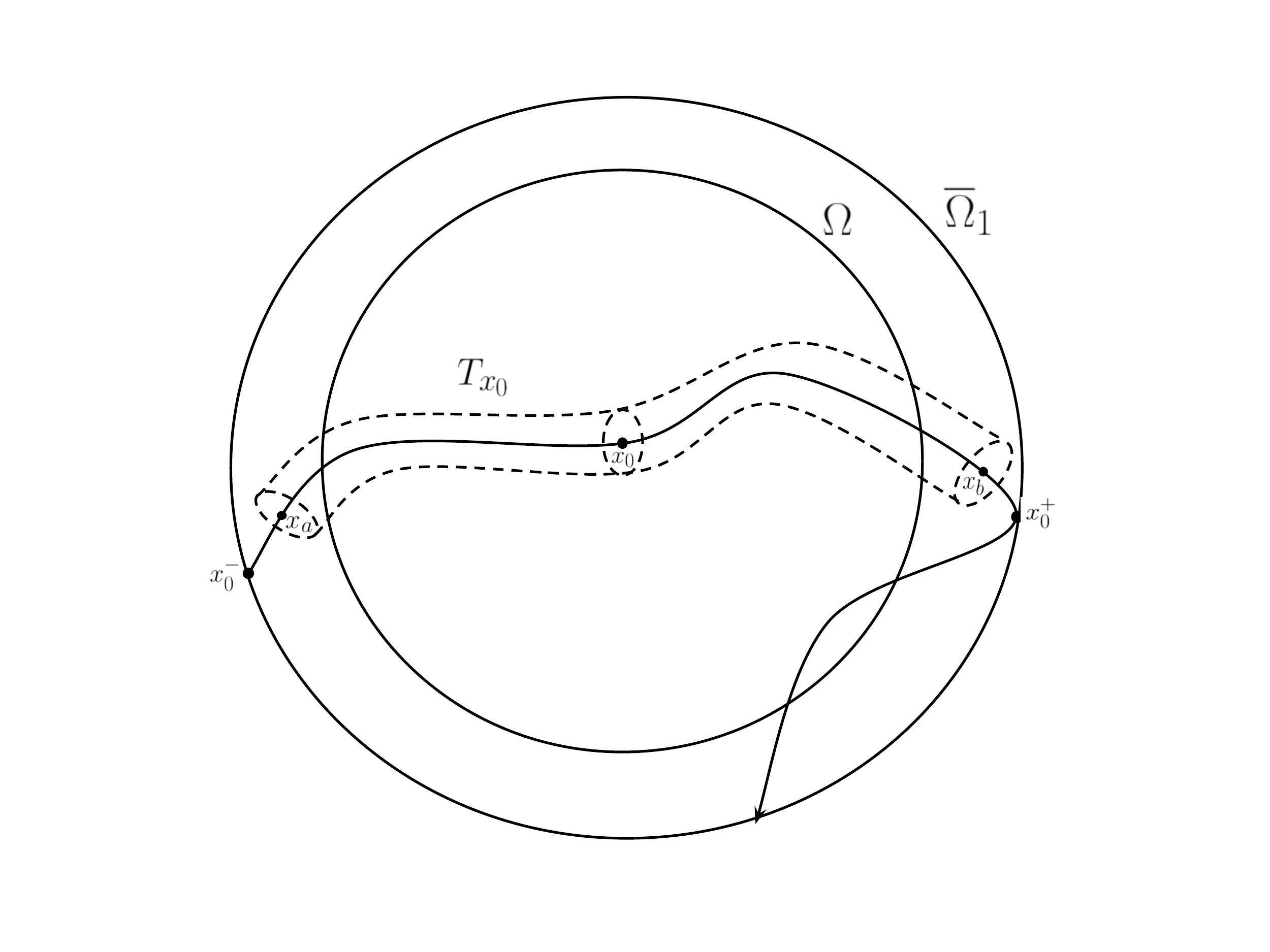}
  \caption{Tubular neighborhood $T_{x_0}$ of integral curve of $\nabla u_0$ from $x_a = x(a)$ to $x_b = x(b)$. }
  \label{fig:Integral_flow}
\end{figure}

Let $x(t): I \to \overline{\Omega}_1$ be a parametrization of the integral curve of $\nabla u_1$ such that $x(0) = x_0$, $\dot{x}(t) = \nabla u_0(x(t))$, and $I$ is the entire interval of definition of the integral curve. Denote by $x_0^+$ the first point on that the integral curve, starting from $x_0$ and traveling in the same direction of the flow,  hits the boundary $\partial\Omega_1$. Similarly denote by $x_0^-$ first point on that the integral curve, starting from $x_0$ and traveling in the opposite direction of the flow hits the boundary $\partial\Omega_1$. We know that $x_0^\pm$ exist because since 
\[
\frac{d}{dt}u(x(t)) = \nabla u_0(x(t))\cdot \dot{x}(t) = \| \nabla u_0(x(t))\|^2> 1/C>0,
\]
then $u(x(t))$ is strictly increasing along the integral curve $x(t)$; and $u$ cannot grow indefinitely in $\overline{\Omega}_1$.  This implies that the integral curve in $\overline{\Omega}_1$ cannot intersect themselves, and   cannot be infinite.

Consider a tubular neighborhood of the integral curve $x(t)$ as $ x_0^-< a \leq t \leq b<x_0^+$, 
\[
T_{x_0} = \{(y',y^n) \in \Omega_1: |y' - y'_0| < \delta_0,\, a \leq t \leq b \},
\]
where $\delta_0 >0$ is  small enough so that $T_{x_0} \cap \{y^n = a \}$ and $T_{x_0} \cap \{y^n = b \}$ are contained in $\Omega_1\setminus\overline{\Omega}$ as shown in Figure \ref{fig:Integral_flow}.  Since $h=0$ in $\Omega_1 \setminus \Omega$, we can write 
\[
h(y', y^n) = \int_{a}^{y^n} (\nabla u_0 \cdot \nabla h)(y',t) \d t \quad \mbox{ for } \quad (y',y^n) \in T_{x_0}. 
\]
Using the Cauchy inequality we get that for $\delta_0 \geq \delta >0$,
\begin{equation} \nonumber 
\begin{split}
\| h(y)\|^2_{L^2(T_{x_0})} & = \int_{|y' - y'_0| < \delta}\int_{a}^{b}  \left| \int_a^{y^n} (\nabla u_0 \cdot \nabla h)(y',t) \d t\right|^2 \d y^n \,\d y'\\
&\leq \int_{|y' - y'_0| < \delta}\int_{a}^{b}   \int_a^{y^n} |(\nabla u_0 \cdot \nabla h)(y',t)|^2 \d t\, \d y^n,\d y'\\
& \leq (b-a) \|\nabla u_0 \cdot \nabla h\|_{L^2(T_{x_0})} \leq 
C\| \nabla u_0 \cdot \nabla h\|.
\end{split}
\end{equation}
We used here the $L^2(T_{x_0})$ norm in the $y$ variables (without the Jacobian coming from the change of the variables) but that norm is equivalent to the original one. 
By the compactness of $\overline{\Omega}$, we can find $T_{x_0}, T_{x_1}, \ldots, T_{x_m}$ such that their union covers $\overline{\Omega}$ and use a partition of unity subordinated to this covering to prove \eqref{ineq:lemma_T0_stability}. 
\end{proof}

We now present the proof for the theorem of conditional stability for the linearized
problem.

\begin{proof}[Proof of Theorem \ref{theor:cond_stability_linearization}]
We first consider the case $p=1$. Let $h \in C^2(\overline{\Omega})$ and denote $\rho = (\sigma - \sigma_0)/\sigma_0 = h/\sigma_0$.  By Lemma \ref{lemma:estimate_T0}, definition of $v_0$, and interpolation estimate in section 4.3.1 in \cite{MR1328645} we have
\begin{equation} \label{ineq:linearization_estimate1}
\begin{split}
\|\rho \| 
&\leq C \|\nabla u_0 \cdot \nabla \rho \|  
\leq C  \|v_0\|_{H^2(\Omega)}  
\leq C  \|v_0\|^{\alpha_1} \cdot \|v_0\|^{1-\alpha_1}_{H^s(\Omega)}.
\end{split}
\end{equation}
Using Proposition~\ref{Prop:LinearDecomp} and Lemma \ref{lemma:stability_for_L}, we also obtain
\begin{equation} \label{ineq:linearization_estimate2}
\begin{split}
\|v_0 \| 
&\leq C \|Lv_0 \| \leq C  \left|\left|\nabla u_0 \cdot \nabla \left(\frac{dF_{\sigma_0}(\rho)}{\sigma_0 \|\nabla u_0 \|  } \right) \right|\right|   \leq C  \|dF_{\sigma_0}(\rho)\|_{H^1(\Omega)}.
\end{split}
\end{equation}
Finally, combining inequalities (\ref{ineq:linearization_estimate1}) and (\ref{ineq:linearization_estimate2}) we proof the theorem in the case $p=1$. For the case $0<p<1$, we use the same reasoning an the better estimate \eqref{ineq:stab-for-L-0<p<1} in Lemma \ref{lemma:stability_for_L} to conclude. 
\end{proof}

We now present the proof of our main result as a consequence of Theorem \ref{stability_appendix} and Theorem \ref{theor:cond_stability_linearization}
\begin{proof}[Proof of Theorem \ref{theor:cond_stability_main_theorem}]
Let $0<\theta <1$, $1>  \beta > \max \{ \theta, 1/2 \} $ and $\alpha_1$ as in \eqref{cond_stability_linearization_p=1}. We apply Theorem \ref{stability_appendix} taking
\[
\begin{array}{cccc}
\B'''_1 = H^{s}(\Omega), & \B''_1 = H^{s_1}(\Omega), & \B_1 = C^2(\Omega), & \B' = L^2(\Omega), 
\end{array}
\]
\[
\B''_2 =  \B'_2 =  \B_2 = H^{1}(\Omega), 
\]
with
\begin{equation}\label{intepolation_values}
(1-\mu_1)s_1  > \frac{n}{2} +2, \quad (1-\mu_2)s_2 =1, \quad (1-\mu_3)s = s_1 , \quad \mbox{for} \quad \mu_1, \mu_2 \in (0,1).
\end{equation}
We choose $0<\mu = \alpha_1 \mu_1\mu_2 < \min\{1/2,\beta\}$ by taking $\mu_1 = \alpha_1$ small enough, we then take $\mu_3$ as 
\[
 1 > \mu_3 = \frac{\beta - \mu}{\beta(1-\mu)} > \frac{1-2\mu}{1-\mu} >0,
\] 
under the penalty of making $s$ large enough.

First notice that as a consequence of \eqref{2-order_linerization} and \eqref{second_order_residue} the differential of $F$ and $\sigma_0$, $\d_{\sigma_0}F$, is a linearization with quadratic remainder as in \eqref{alpha_order_linearization}. Second, conditional stability for the linearizion is consequence of Theorem \ref{theor:cond_stability_linearization}, 
with $\alpha_1 = 1$ in the case $0<p<1$ and $0<\alpha_1<1$ for $p=1$. Notice that $s_1 = \frac{n + 4}{2(1-\alpha_1)} > \frac{2}{1-\alpha_1}$. Third, interpolation estimates follow by \eqref{intepolation_values}. Finally, continuity of $\d F_{\sigma_0}: C^2(\overline{\Omega}) \to H^1(\Omega)$ follows by \eqref{dF} and \eqref{ineq:vt_estimate}. Hence by Theorem \ref{stability_appendix}, for any $L>0$ there exist $\epsilon>0$ and $C>0$, so that for any $\sigma$ with 
\[
\|\sigma - \sigma_0\|_{C^2(\overline{\Omega})} < \epsilon, \qquad  \|\sigma \|_{H^s(\overline{\Omega})} \leq L,
\]
one has
\[
\|\sigma - \sigma_0 \|_{C^2(\overline{\Omega})} \leq C \| F(\sigma) - F(\sigma_0)\|_{H^1(\Omega)}^{\beta} < C \| F(\sigma) - F(\sigma_0)\|_{H^1(\Omega)}^{\theta}.
\]
which proofs \eqref{ineq:main_theorem_non-linear_stability}.
\end{proof}

\appendix

\section{Stability of non-linear inverse problems by linearization} \label{stability_appendix}

The following conditional stability Theorem through linearization is a generalization of Theorem 2 in \cite{Stefanov20092842}. 

\begin{theorem} \label{non-linear_stability_theorem}
Let $F: \B_1 \to \B_2$ be a continuous non-linear map between two Banach spaces.  Assume the there exist Banach spaces $\B_1''' \subset \B_1'' \subset \B_1 \subset \B_1'$ and $\B_{2}'' \subset \B_{2}' \subset \B_2$ that satisfy the following:
\begin{enumerate}
\item \textbf{$\alpha$-order linearization}: for $\sigma_0 \in \B_1$ there exist $\d F_{\sigma_0}:\B_1 \to \B_2$ linear map and $\alpha >1$ such that
\begin{equation}\label{alpha_order_linearization}
F(\sigma) = F(\sigma_0) +  \d F_{\sigma_0}(\sigma - \sigma_0) + R_{\sigma_0}(\sigma - \sigma_0),
\end{equation}
with $\|R_{\sigma_0},(\sigma - \sigma_0)\|_{\B_2} \leq C_{\sigma_0}\|\sigma - \sigma_0\|^{\alpha}_{\B_1}$, for $\sigma$ in some $\B_1$-neighborhood of $\sigma_0$. We say that $\d F_{\sigma_0}$ is the differential of $F$ at $\sigma_0$ with remainder of order $\alpha$.
\item \textbf{conditional stability of linearization}: there exist $C>0$ such that
\begin{equation}\nonumber 
\|h\|_{\B_1'} \leq C \|\d F_{\sigma_0} h\|^{\alpha_1}_{\B_2'}\|h\|^{1- \alpha_1}_{\B_1''} \quad \mbox{ for } \quad \alpha_1 \in (0,1].
\end{equation}
\item \textbf{interpolation estimates}: there exist $C>0$ such that 
\begin{equation} \nonumber 
\|g\|_{\B_2'} \leq C \|g \|^{\mu_2}_{\B_2}\|g \|^{1-\mu_2}_{\B_2''}, \quad \|h\|_{\B_1} \leq C \|h \|^{\mu_1}_{\B_1'}\|h \|^{1-\mu_1}_{\B_1''}, \quad \|h\|_{\B_1''} \leq C \|h \|^{\mu_3}_{\B_1}\|h \|^{1-\mu_3}_{\B_1'''}
\end{equation}
for  $ \mu_1, \mu_2 \in (0,1]$ and $1 \geq \mu_3 \geq \max\{0,(1- \alpha\mu)/(1- \mu)\}$ where $\mu = \alpha_1 \mu_1 \mu_2$.
\item \textbf{continuity of $\d F_{\sigma_0}$}: the differential $\d F_{\sigma_0}$ is continuous from $\B_1''$ to $\B_2''$.
\end{enumerate}
Then we have local conditional stability. For any $L>0$ there exist $\epsilon>0$ and $C>0$, so that for any $\sigma$ with 
\begin{equation} \nonumber 
\|\sigma - \sigma_0\|_{\B_1} < \epsilon, \qquad  \|\sigma \|_{\B_1'''} \leq L,
\end{equation} 
one has
\begin{equation}  \label{weak_cond_stability}
\|\sigma - \sigma_0 \|_{\B_1} \leq C \| F(\sigma) - F(\sigma_0)\|_{\B_2}^{\beta}.
\end{equation}
where $\beta = \mu / (1-\mu_3(1-\mu))$. In particular one has Lipschitz stability (i.e., $\beta = 1$) when $\mu_3 =1$, this happens for example when $\B'''_{1} = \B'_{1}$.
\end{theorem}

\begin{proof}
Let $L>0$, we use the H\"older inequality $(a + b)^{\eta} \leq a^{\eta} + b^\eta$ for $a, b\geq 0 $ and $0<\eta<1$. the following inequalities follow easily from the hypothesis
\begin{equation*}
\begin{split}
\|\sigma - \sigma_0 \|_{\B_1} &\leq C\|\sigma -\sigma_0 \|^{\mu_1}_{\B_1'}\|\sigma -\sigma_0 \|^{1-\mu_1}_{\B_1''}\\
& \leq C \|\d F_{\sigma_0}(\sigma -\sigma_0)\|^{\mu_1\alpha_1}_{\B_2'} \cdot \|\sigma -\sigma_0 \|^{1 -\alpha_1\mu_1}_{\B_{1}''} \\ 
& \leq C  \|\d F_{\sigma_0}(\sigma -\sigma_0) \|_{\B_2}^{\mu}\cdot \|\d F_{\sigma_0}(\sigma -\sigma_0) \|_{\B_2''}^{\alpha_1 \mu_1(1-\mu_2)} \cdot \|\sigma -\sigma_0 \|^{1 -\alpha_1\mu_1}_{\B_{1}''} \\
& \leq C \left( \|F(\sigma) - F(\sigma_0) \|_{\B_2} + C_{\sigma_0} \| \sigma - \sigma _0\|^{\alpha}_{\B_1} \right)^{\mu}  \cdot \|\sigma -\sigma_0 \|^{1 -\mu}_{\B_{1}''} \\ 
& \leq  C \cdot L^{(1-\mu_3)(1-\mu)} \left( \|F(\sigma) - F(\sigma_0) \|_{\B_2}^{\mu} + C_{\sigma_0} \| \sigma - \sigma _0\|_{\B_1}^{\alpha\mu} \right)  \cdot \|\sigma -\sigma_0 \|^{\mu_3(1-\mu)}_{\B_{1}}. 
\end{split}
\end{equation*}
Hence we obtain
\begin{equation*}
\|\sigma - \sigma_0 \|^{1-\mu_3(1-\mu)}_{\B_1} (1 - C_{\sigma_0}\|\sigma - \sigma_0 \|^{\mu_3(\mu-1) + \alpha\mu -1}_{\B_1}) \leq C \|F(\sigma) - F(\sigma_0) \|_{\B_2}^{\mu}
\end{equation*}
by hypothesis $\mu_3(1-\mu) + \alpha\mu -1   \geq 0$ then there exist $\epsilon >0$ so that (\ref{weak_cond_stability}) holds.
\end{proof}

\bibliographystyle{amsplain}

\end{document}